\documentclass[11pt]{article}
\usepackage{}
\usepackage{amsfonts}
\usepackage{mathrsfs}
\usepackage{amssymb,amsmath,mathrsfs, amsthm}
\usepackage{palatino,cite}
\usepackage{graphicx}
\usepackage{mathtools}
\usepackage{hyperref,cases,anysize,enumerate}

\numberwithin{equation}{section}

\newtheorem{theorem}{Theorem}[section]
\newtheorem{lemma}{Lemma}[section]
\newtheorem{corollary}{Corollary}[section]
\newtheorem{proposition}{Proposition}[section]
\theoremstyle{definition}
\newtheorem{definition}{Definition}[section]
\newtheorem{remark}{Remark}[section]
\newtheorem{example}{Example}[section]
\theoremstyle{remark}

\date{}
\begin{document}

\title{On $U(n)$-invariant strongly convex complex Finsler metrics}
\author{Kun Wang(wangkunmath@126.com)\\
School of Mathematical Sciences, Xiamen
University\\ Xiamen 361005, China\and
Hongchuan Xia (xhc@xynu.edu.cn)\\
College of Mathematics and Statistics,
Xinyang Normal University\\
Xinyang 464000, China\and
Chunping Zhong (zcp@xmu.edu.cn)
\\
School of Mathematical Sciences, Xiamen
University\\ Xiamen 361005, China
 }

\date{}
\maketitle
\begin{abstract}
In this paper, we obtain a necessary and sufficient condition for a $U(n)$-invariant complex Finsler metric $F$ on domains in $\mathbb{C}^n$ to be  strongly convex, which also makes it possible to investigate relationship between real and complex Finsler geometry via concrete and computable examples. We prove a rigid theorem which states  that
a $U(n)$-invariant strongly convex complex Finsler metric $F$ is a real Berwald metric if and only if $F$ comes from a $U(n)$-invariant Hermitian metric. We give a characterization of $U(n)$-invariant weakly complex Berwald metrics with vanishing holomorphic sectional curvature and obtain an explicit formula for holomorphic curvature of $U(n)$-invariant strongly pseudoconvex complex Finsler metric. Finally, we  prove that the real geodesics of some $U(n)$-invariant complex Finsler metric restricted on the unit sphere $\pmb{S}^{2n-1}\subset\mathbb{C}^n$ share a specific property as that of the complex Wrona metric on $\mathbb{C}^n$.
\end{abstract}
\textbf{Keywords:} $U(n)$-invariant complex Finsler metric, strongly convex, real Berwald metric, holomorphic curvature.\\
\textbf{MSC(2010):}  53C60, 53C40.\\

\section{Introduction and main results}
As is well know, a Hermitian metric on a complex manifold $M$ is a Riemannian metric which is compatible with the complex structure $J$ of $M$.
In Finsler geometry (especially when the metric is not quadratic), however, real and complex Finsler geometry are not as tightly related as that of Riemannian and Hermitian geometry. Usually the differential geometry of real Finsler metrics requires the metrics  to be strongly convex along real tangent directions, while the differential geometry of complex Finsler metrics only requires the metrics to be strongly pseudoconvex along complex tangent directions, we refer to \cite{BCS,Sh,AP} for more details.

 A real Finsler metric is not necessary reversible while a complex Finsler metric is always reversible. A strongly pseudoconvex complex Finsler metric on a complex manifold is  not necessarily a real Finsler metric, and vice versa. A complex Finsler metric on a complex manifold is called strongly convex if it is also a real Finsler metric \cite{AP}. The first fundamental example of complex Finsler metrics which are smooth outside the zero section of the holomorphic tangent bundle of a complex manifold is undoubtedly the Kobayashi and Carath$\acute{\mbox{e}}$odory metrics on bounded strictly convex domains with smooth boundaries in $\mathbb{C}^n$, where these two holomorphic invariant metrics coincide and are strongly convex weakly K\"ahler Finsler metrics with constant holomorphic sectional curvature $-4$ \cite{LE1}.  Even in this special case, however, we don't have the explicit formulae for the Kobayashi and Carath$\acute{\mbox{e}}$odory metrics.

 %As wrote by H.-H. Wu and F. Zheng in their article \cite{WZ}, "perhaps one of the major reasons that made the problem (\emph{Green-Wu-Mok-Siu-Yau Conjecture}) so resilient is the lack of examples. While it is relatively easy to write down examples of positively curved complete Riemannian metrics on $\mathbb{R}^n$, it is much more difficult to write down (explicitly or simply to show the existence) positively curved complete K\"ahler metrics on $\mathbb{C}^n$ when $n>1$."
 In \cite{WZ}, H.-H. Wu and F. Zheng gave a systematic study of $U(n)$-invariant K\"ahler metrics on $\mathbb{C}^n$ with positive bisectional or sectional curvatures, and proved that the set $\mathcal{M}_n$ of all complete $U(n)$-invariant K\"ahler metrics on $\mathbb{C}^n$ with positive bisectional curvatures is actually quite large.
 Prior to their paper \cite{WZ}, there are only three example in this direction which were given in Klembeck \cite{K}, and H.-D. Cao \cite{C1,C2}, all of which are $U(n)$-invariant K\"ahler metrics on $\mathbb{C}^n$.

In complex Finsler geometry, prior to the  article \cite{Zh1}, there is few methods to construct even strongly pseudoconvex complex Finsler metrics which are not Hermitian quadratic. A natural question in complex Finsler geometry one may ask is whether there are $U(n)$-invariant K\"ahler Finsler metrics (in the sense of M. Abate and Patrizio\cite{AP}) which are not Hermitian quadratic. In \cite{Zh1}, the author obtained the necessary and sufficient condition for a $U(n)$-invariant complex Finsler metric to be strongly pseudoconvex, and proved that among all $U(n)$-invariant strongly pseudoconvex complex Finsler metrics on domains $D\subseteq\mathbb{C}^n$ there is no K\"ahler Finsler metrics or complex Berwald metrics other than K\"ahler metrics or Hermitian metrics. But fortunately, we found that there are lots of $U(n)$-invariant weakly complex Berwald metrics in the sense of  \cite{Zh2} which are not Hermitian quadratic. The complex geodesic spray coefficients of these metrics are quadratic and holomorphic with respect to complex tangent directions.

 In complex Finsler geometry, especially when dealing with relationship between real and complex Finsler metrics, it is natural to assume that the complex Finsler metrics considered are strongly convex so that we can use some geometric notions (such as flag curvature or other non-Riemannian quantities) from real Finsler geometry  \cite{Zh2,XZ4}. So far to our knowledge, there are few such examples and even no effective methods to construct strongly convex complex Finsler metrics which are not Hermitian quadratic.

Recently, we systematically investigated $U(n)$-invariant complex Finsler metrics on domains $D\subseteq\mathbb{C}^n$,  and the general complex $(\pmb\alpha, \pmb\beta)$ metrics on complex manifolds. We showed that there are lots of strongly pseudoconvex (even strongly convex) complex Finsler metrics \cite{Zh1, XZ1, XZ2, XZ3} which are not Hermitian quadratic. A $U(n)$-invariant complex Finsler metric $F$ defined on a domain $D\subseteq\mathbb{C}^n$ is a complex Finsler metric $F$ which is invariant under the action of the unitary group $U(n)$ in the sense that
$$F(Az, Av)=F(z,v),\quad \forall z\in D, v\in T_z^{1,0}D,\forall A\in U(n).$$

Let $F$ be a $U(n)$-invariant complex Finsler metric on a domain $D\subseteq\mathbb{C}^n$.
Denote $\langle \cdot,\cdot\rangle$ the canonical complex Euclidean inner product and $\|\cdot\|$ the induced norm on $\mathbb{C}^n$, respectively.
It was proved in H. Xia and C. Zhong \cite{XZ1} that $F$ can be expressed as
\begin{equation}
F(z,v)=\sqrt{r\phi(t,s)},\label{uicf}
\end{equation}
where
\begin{equation}
r=\|v\|^{2},\quad t=\|z\|^{2},\quad s=\frac{|\langle z,v\rangle|^{2}}{\|v\|^{2}},\quad \forall z\in D, \;0\neq v\in T_z^{1,0}D   \label{rst}
\end{equation}
and $\phi(t,s):[0,+\infty)\times[0,+\infty)\rightarrow (0,+\infty)$ is a smooth function. Note that we are only interested in the case $0\neq v\in T_z^{1,0}D$ since by Lemma 2.3. 1 in \cite{AP}, a strongly pseudoconvex complex Finsler metric $F$ is smooth over the zero section of $T^{1,0}D$ if and only if $F$ comes from a Hermitian metric on $D$.

In \cite{Zh1}, the third author obtained a sufficient condition for $U(n)$-invariant complex Finsler metrics of the form \eqref{uicf} to be strongly convex along real tangent directions (see Proposition \ref{p4}). In this paper, as our first main result, we obtain the sufficient and necessary condition for $U(n)$-invariant complex Finsler metrics of the form \eqref{uicf} to be strongly convex along real tangent directions.
\begin{theorem}\label{th1}\ \
Let $F(z,v)=\sqrt{r\phi(t,s)}$ be a $U(n)$-invariant complex Finsler metric on a domain $D \subseteq \mathbb{C}^n$. Then $F$ is a strongly convex complex Finsler metric if and only if
 \begin{align*}
 &\phi-s\phi_s>0,\\
 &\phi+(t-s)\phi_s>0,\\
 &(\phi-s\phi_s)[\phi+(t-s)\phi_s]+2s(t-s)\phi \phi_{ss}> 0
 \end{align*}
for every $z\in D$ and any nonzero vector $v\in T_z^{1,0}D$.
\end{theorem}
\begin{remark}
There are lost of functions $\phi(t,s)$ which satisfy the conditions in Theorem \ref{th1}.
\end{remark}
As an application of Theorem \ref{th1}, we give an example of $U(n)$-invariant complex Finsler metric which is strongly pseudoconvex, but not strongly convex (see Example \ref{e2}).

 Our recent results \cite{XZ1, XZ2} show that among $U(n)$-invariant complex Finsler metrics, there are lots of weakly complex Berwald metrics which do not come from complex Berwald metrics. Thus, it is natural to ask that among $U(n)$-invariant strongly convex complex Finsler metrics, whether there are real Berwald metrics? Our second theorem gives a negative answer to this question.
\begin{theorem}\ \
Let $F(z,v)=\sqrt{r\phi(t,s)}$ be a  $U(n)$-invariant strongly convex complex Finsler metric on a domain $D \subseteq \mathbb{C}^n$. Then $F$ is a real Berwald metric if and only if $F$ comes from a $U(n)$-invariant Hermitian metric.
\end{theorem}

The following theorem gives a characterization of $U(n)$-invariant weakly complex Berwald metrics with vanishing holomorphic sectional curvature.

\begin{theorem}\label{vc}
Suppose that $F(z,v)=\sqrt{r\phi(t,s)}$ is a $U(n)$-invariant strongly pseudoconvex complex Finsler metric on a domain $D\subseteq\mathbb{C}^n$. Then  $F$ is a weakly complex Berwald metric with vanishing holomorphic sectional curvature if and only if $F=\sqrt{rf(s-t)}$ for some smooth positive function $f(w)$ with $w=s-t$.
\end{theorem}

The following theorem gives an explicit formula for the holomorphic sectional curvature of any $U(n)$-invariant strongly pseudoconvex complex Finsler metrics.
\begin{theorem}\ \
Let $F(z,v)=\sqrt{r\phi(t,s)}$ be a $U(n)$-invariant strongly pseudoconvex complex Finsler metric. Then the holomorphic curvature $K_F(z,v)$ of $F$ along $0\neq v\in T_z^{1,0}D$ is given by
\begin{eqnarray*}
K_F(z,v)&=&-\frac{2}{\phi^2k_1}\bigg\{k_1\Big[s(\phi_{tt}+2\phi_{st}+\phi_{ss})+(\phi_t+\phi_s)\Big]-s^2(t-s)\phi(\phi_{st}+\phi_{ss})^2\nonumber\\
 &&+2s^2(t-s)\phi_s(\phi_{st}+\phi_{ss})(\phi_t+\phi_s)-s\Big[c_0+(t-s)\phi_s+s(t-s)\phi_{ss}\Big](\phi_t+\phi_s)^2\bigg\},\nonumber
\end{eqnarray*}
where $c_0$ and $k_1$ are given by \eqref{ck}.
Especially, if $0\in D$, then at the origin the holomorphic curvature of $F$ along any nonzero tangent vector $v\in T_0^{1,0}D$ is a constant, i.e.,
 $$
 K_F(0,v)=\frac{-2[\phi_t(0,0)+\phi_s(0,0)]}{\phi^2(0,0)}=\mbox{constant},\qquad \forall \; 0 \neq  v\in T_0^{1,0}D.
 $$
\end{theorem}

The above theorem shall be found useful in seeking $U(n)$-invariant complex Finsler metrics with some specific holomorphic sectional curvatures. If the origin of $\mathbb{C}^n$ is contained in $D$, the above theorem also shows that the holomorphic sectional curvature of any $U(n)$-invariant strongly pseudoconvex complex Finsler metric is constant at the origin in any nonzero tangent direction.

Our last result focus on the real geodesics of any $U(n)$-invariant complex  Finsler metrics.
\begin{theorem}
  Suppose that $F(z,v)=\sqrt{r\phi(t,s)}$ is a $U(n)$-invariant complex Finsler metric defined on $\mathbb{C}^n$ which is normalized such that $\phi(1,0)=1$. Let $0<\alpha<\frac{\pi}{2}$ and $\gamma(\tau)(0\leq \tau\leq\alpha)$ be a real geodesic of $F$ on the unit sphere $\pmb{S}^{2n-1}\subset\mathbb{C}^n$ which is parameterized by arc length. Then
$$
L(\gamma)=\alpha.
$$
\end{theorem}
The above theorem shows that as far as the lengths of real geodesics are concerned, when restricted to the unit sphere in $\mathbb{C}^n$, any $U(n)$-invariant complex Finsler metric on $\mathbb{C}^n$ normalized by $\phi(1,0)=1$ shares the same property as that of the complex Wrona metric on $\mathbb{C}^n$ \cite{SR}.

The remainder of this paper is organized as follows.
In section \ref{sec2}, we introduce some definitions and notions which are needed in this paper.
In section \ref{sec3}, we derive the real fundamental tensor and its inverse of a $U(n)$-invariant strongly convex complex Finsler metric. In section \ref{sec4}, we give a necessary and sufficient condition for a $U(n)$-invariant complex Finsler metric $F$ to be a strongly convex complex Finsler metric.
In section \ref{sec5}, we prove that among $U(n)$-invariant strongly convex complex Finsler metrics there is no real Berwald metric other than $U(n)$-invariant Hermitian metrics.
In section \ref{sec6}, we derive an explicit formula for the holomorphic sectional curvature for an arbitrary $U(n)$-invariant strongly pseudoconvex complex Finsler metric.
In section \ref{sec7}, we investigate the real geodesics of a $U(n)$-invariant complex Finsler metric on the unit sphere $\pmb{S}^{2n-1}\subset\mathbb{C}^n$.

\section{Preliminaries}\label{sec2}
Let us recall some definitions and notions which are needed in this paper. We refer to \cite{AP} for more details.
Let $M$ be a complex $n$-dimensional manifold with the canonical complex structure $J$. We denote $T_{\mathbb{R}}M$ as the real tangent
bundle and $T_{\mathbb{C}}M$ as the complexified tangent bundle of $M$. Then, $J$ acts in a complex linear manner on $T_{\mathbb{C}}M$ such that $T_{\mathbb{C}}M = T^{1,0}M\oplus T^{0,1}M$,
 where $T^{1,0}M$ is called the holomorphic tangent bundle of $M$. $T^{1,0}M$ is a complex manifold of complex dimension $2n$,
and we also denote $J$ as the induced complex structure on $T^{1,0}M$ if it causes no confusion.
Let $\{z^1,\cdots, z^n\}$ be a set of local complex coordinates on $M$, with $z^\alpha = x^\alpha+\sqrt{-1}x^{\alpha+n}$, such that $\{x^1,\cdots,x^n,x^{1+n}, \cdots, x^{2n}\}$
are local real coordinates on $M$. Denote $\{z^1,\cdots,z^n,v^1,\cdots,v^n\}$ as the induced complex coordinates on $T^{1,0}M$, with $v^\alpha=u^\alpha +\sqrt{-1}u^{\alpha+n}$,
such that $\{x^1,\cdots, x^{2n}, u^1,\cdots, u^{2n}\}$ are local real coordinates on $T_{\mathbb{R}}M$.

In the following, we denote $\tilde{M}$ as the complement of the zero section in $T_{\mathbb{R}}M$ or $T^{1,0}M$, depending on whether it is the real or complex situation. And the Einstein summation convention is assumed throughout this paper.

The bundles $T^{1,0}M$ and $T_{\mathbb{R}}M$ are isomorphic. We choose the explicit isomorphism $^o: T^{1,0}M\rightarrow T_{\mathbb{R}}M$ with its inverse $_o: T_{\mathbb{R}}M\rightarrow T^{1,0}M$, respectively, which are given by
\begin{equation*}
T_{\mathbb{R}}M\ni u^a\frac{\partial}{\partial x^a}=u=v^o=v+\overline{v},\;\;\;\;\forall\; v=v^\alpha\frac{\partial}{\partial z^\alpha}\in T^{1,0}M
\end{equation*}
and
\begin{equation*}
T^{1,0}M\ni v^\alpha\frac{\partial}{\partial v^\alpha}=v=u_o=\frac{1}{2}(u-\sqrt{-1}Ju),\;\;\;\;\forall\; u=u^a\frac{\partial}{\partial x^a}\in T_{\mathbb{R}}M.
\end{equation*}
\begin{definition}[\cite{AP}]\ \
A real Finsler metric on a manifold $M$ is a function $F: T_{\mathbb{R}}M\rightarrow \mathbb{R}^+$ that satisfies
the following properties:
\begin{enumerate}[(a)]
\item $G=F^2$ is smooth on $\tilde{M}$;

\item $F(p,u)>0$ for $(p,u)\in\tilde{M}$;

\item $F(p,\lambda u)=|\lambda|F(p,u)$ for all $(p,u)\in T_{\mathbb{R}}M$ and $\lambda\in\mathbb{R}$;

\item For any $p\in M$, the indicatrix $I_F(p)=\{u\in T_pM|\;F(p,u)<1$\} is strongly convex.
\end{enumerate}
\end{definition}
Note that condition (d) is equivalent to the following matrix
\begin{equation*}
(G_{ij}):=\left(\frac{\partial^2G}{\partial u^i\partial u^j}\right)
\end{equation*}
being positive definite on $\tilde{M}$.
\begin{definition}[\cite{AP}]\ \
A complex Finsler metric $F$ on a complex manifold $M$ is a continuous function $F: T^{1,0}M\rightarrow \mathbb{R}^+$ that satisfies:
\begin{enumerate}[(i)]
\item $G=F^2$ is smooth on $\tilde{M}$;

\item $F(p,v)>0$ for all $(p,v)\in\tilde{M}$;

\item $F(p,\zeta v)=|\zeta|F(p,v)$ for all $(p,v)\in T^{1,0}M$ and $\zeta\in\mathbb{C}$.

\end{enumerate}
\end{definition}

\begin{definition}[\cite{AP}]\ \
A complex Finsler metric $F$ is called strongly pseudoconvex
if the Levi matrix
$$(G_{\alpha\overline{\beta}}):=\left(\frac{\partial^2G}{\partial v^\alpha\partial\overline{v^\beta}}\right)$$
is positive definite on $\tilde{M}$.
\end{definition}

Let $F: T^{1,0}M\rightarrow \mathbb{R}^+$ be a complex Finsler metric on a complex manifold $M$. Using the complex structure $J$ on $M$ and the bundle map $_o: T_{\mathbb{R}}M\rightarrow T^{1,0}M$, we can define a real function
\begin{equation*}
F^o: T_{\mathbb{R}}M\rightarrow \mathbb{R}^+,\;\;F^o(u):=F(u_o),\;\;\forall\; u\in T_{\mathbb{R}}M.
\end{equation*}
\begin{definition}[\cite{AP}]\ \
A complex Finsler metric $F$ is called strongly convex if the associated function $F^o$ is a real Finsler metric.
\end{definition}
For a strongly convex complex Finsler metric $F$, we use the same symbol $F$ to denote the associated real Finsler metric $F^o$, where we understand that $F(u)$ is defined by $F(u_o)$ for $u\in T_{\mathbb{R}}M$.

In the following, as in \cite{AP}, for functions defined on $\tilde{M}$, the lower Greek indices like $\alpha,\beta$ and so on run from $1$ to $n=\dim_{\mathbb{C}}M$ while lower Latin indices like $i,j$ and so on run from $1$ to $2n=\dim_{\mathbb{R}}M$, and subscripts denote derivatives. We use a semi-colon to distinguish between derivatives with respect to the complex coordinates $z=(z^1,\cdots,z^n)$ and derivatives with respect to complex vector coordinates $v=(v^1,\cdots,v^n)$. We also use semi-colon to distinguish between derivatives with respect to the base manifold coordinates $x=(x^1,\cdots,x^{2n})$ and derivatives with respect to the vector variables $u=(u^1,\cdots,u^{2n})$; for example
\begin{eqnarray*}
G_{\alpha}&=&\frac{\partial G}{\partial v^\alpha},\quad G_{;\mu}=\frac{\partial G}{\partial z^\mu},\quad G_{\alpha;\overline{\gamma}}=\frac{\partial^2G}{\partial\overline{z^\gamma}\partial v^\alpha},\\
G_{;i}&=&\frac{\partial G}{\partial x^i},\quad \,\;G_a=\frac{\partial G}{\partial u^a},\;\quad G_{a;i}=\frac{\partial^2G}{\partial u^a\partial x^i}.
\end{eqnarray*}

For a strongly pseudoconvex complex Finsler metric $F$, we denote $\mathbb{G}^\gamma$ the complex geodesic spray coefficients associated to $F$ which are given by
$$
\mathbb{G}^\gamma=\frac{1}{2}\varGamma_{;\mu}^\gamma v^\mu,\quad \varGamma_{;\mu}^\gamma=G^{\overline{\tau}\gamma}G_{\overline{\tau};\mu},
$$
where $(G^{\overline{\tau}\gamma})$ is the inverse matrix of $(G_{\alpha\overline{\tau}})$ such that $G_{\alpha\overline{\tau}}G^{\overline{\tau}\gamma}=\delta_\alpha^\gamma$.

For a strongly convex complex Finsler metric $F$, we denote $\textbf{G}^i$ the real geodesic spray coefficients associated to $F$ which are give by
\begin{equation*}
\textbf{G}^i=\frac{1}{4}g^{ij}(G_{j;k}u^k-G_{;j}),
\end{equation*}
where $(g^{ij})$ is the inverse matrix of $(g_{jk})=\left(\frac{1}{2}G_{jk}\right)$ such that $g^{ij}g_{jk}=\delta^i_k$.

It is known that a real geodesic $x=(x^1(t),\cdots,x^n(t),x^{n+1}(t),\cdots,x^{2n}(t)):[0,1]\rightarrow M$ on a strongly convex complex Finsler manifold $(M,F)$ satisfies the following system of equations
$$
\frac{d^2x^i(t)}{dt^2}+2\textbf{G}^i(x,u)=0,\quad i=1,\cdots,2n
$$
where $u=(\frac{dx^1(t)}{dt},\cdots,\frac{dx^{2n}(t)}{dt})$.

\begin{definition}[\cite{CS}]\ \
A real Finsler metric $F$ on a manifold $M$ is a real Berwald metric if in any standard local coordinate system $(x, u)$ in $T_{\mathbb{R}}M$, the real spray coefficients $\textbf{G}^i$ are quadratic with respect to $u\in T_xM$ for any $x\in M$.
\end{definition}

Real Berwald metrics form an important class of real Finsler metrics, which are natural generalization of Riemannian metrics in real Finsler geometry \cite{BCS, CS, Sh}. One may ask whether there exists strongly convex complex Finsler metric which is simultaneously a real Berwald metric. In \cite{Zh2}, the third author proved that a strongly convex K\"ahler Finlser metric is a complex Berwald metric if and only if it is a real Berwald metric. Let $\pmb{\alpha}^2(\xi)=a_{i\overline{j}}(z)\xi^i\overline{\xi^j}$ and $\pmb{\beta}^2=b_{i\overline{j}}(w)\eta^i\overline{\eta^j}$ be two Hermitian metrics on complex manifolds $M_1$ and $M_2$, respectively. Recently in \cite{XZ4},we proved that the following Szab$\acute{\mbox{o}}$ metric
$$
F_\varepsilon=\sqrt{\pmb{\alpha}^2(\xi)+\pmb{\beta}(\eta)+\varepsilon(\pmb{\alpha}^{2k}(\xi)+\pmb{\beta}^{2k}(\eta))^{\frac{1}{k}}},\quad \varepsilon\in(0,+\infty)
$$
is actually a strongly convex complex Berwald metric on $M=M_1\times M_2$. Moreover, $F_\varepsilon$ is a strongly convex K\"ahler Finsler metric on  $M=M_1\times M_2$ if $\pmb{\alpha}^2(\xi)$ and $\pmb{\beta}^2(\eta)$ are two K\"ahler metrics on $M_1$ and $M_2$, respectively. This provides us with an important class of strongly convex complex Finsler metrics which are both K\"ahler Finsler metrics and real Berwald metrics on $M=M_1\times M_2$.

A natural question one may ask is whether there exists $U(n)$-invariant strongly convex complex Finsler metric which is also a real Berwald metric. In \cite{Zh1}, however, the second author prove that there does not exist any non-Hermitian $U(n)$-invariant K\"ahler Finsler metric on domains $D\subset\mathbb{C}^n$. Thus in order to answer the above question, we need to investigate the real Finsler geometry of $U(n)$-invariant strongly convex complex Finsler metrics.

\section{Real fundamental tensor of a $U(n)$-invariant strongly convex complex Finsler metric}\label{sec3}

In \cite{Zh1}, the second author proved the following proposition (See Proposition 2.6 and Remark 2.7 in \cite{Zh1}).
\begin{proposition} Let $F=\sqrt{r\phi(t,s)}$ be a $U(n)$-invariant metric on a domain $D\subseteq\mathbb{C}^n$ such that $n\geq 2$. Then $F$ is a strongly psuedoconvex complex Finsler metric if and only if either
\begin{equation}
\phi-s\phi_s>0\quad\mbox{and}\quad (\phi-s\phi_s)[\phi+(t-s)\phi_s]+s(t-s)\phi\phi_{ss}>0\label{dg3}
\end{equation}
whenever $n\geq 3$, or
\begin{equation}
(\phi-s\phi_s)[\phi+(t-s)\phi_s]+s(t-s)\phi\phi_{ss}>0\label{de2}
\end{equation}
whenever $n=2$  for any $z\in D$ and any nonzero vectors $v\in T_z^{1,0}D$.
\end{proposition}
\begin{proof}
In this paper, we give a proof  which is different from  that of given in \cite{Zh1}. Denote
$$c_0:=\phi-s\phi_s,\quad H:=(G_{\alpha\overline{\beta}})$$
and
\begin{equation}
B:=\begin{pmatrix}
     s_1 & \overline{z^1} \\
     \vdots & \vdots \\
     s_n & \overline{z^n} \\
   \end{pmatrix},\quad
   X:=\begin{pmatrix}
        r\phi_{ss} & 0 \\
        0 & \phi_s \\
      \end{pmatrix},
      \quad
   B^\ast:=\begin{pmatrix}
             s_{\overline{1}} & \cdots & s_{\overline{n}} \\
             z^1 & \cdots & z^n \\
           \end{pmatrix}.
           \label{nt}
\end{equation}
Then by (2.16) in \cite{Zh1},
\begin{equation}
H=c_0I_n+BXB^\ast.
\end{equation}
By Lemma 4.1 in \cite{XZ3}, we have
\begin{eqnarray}
\det(\lambda I_n-H)=\det((\lambda-c_0)I_n-BXB^\ast)&=&(\lambda-c_0)^{n-2}\det((\lambda-c_0)I_2-B^\ast BX).\label{ev}
\end{eqnarray}
Using \eqref{nt}, it is easy to check that
\begin{equation}
B^\ast BX=\begin{pmatrix}
            s(t-s)\phi_{ss} & \frac{\overline{\langle z,v\rangle}}{r}(t-s)\phi_s \\
            \langle z,v\rangle (t-s)\phi_{ss} & t\phi_s \\
          \end{pmatrix}.
          \label{bsbx}
\end{equation}
Substituting \eqref{bsbx} into \eqref{ev}, we get
\begin{eqnarray*}
\det(\lambda I_n-H)&=&(\lambda-c_0)^{n-2}\Big\{\lambda^2-[2c_0+t\phi_s+s(t-s)\phi_{ss}]\lambda+k_1\Big\},
\end{eqnarray*}
where
\begin{equation}
k_1=(\phi-s\phi_s)[\phi+(t-s)\phi_s]+s(t-s)\phi\phi_{ss}.
\end{equation}

For $n=2$, we denote the eigenvalue of $(G_{\alpha\overline{\beta}})$ by $\lambda_1$ and $\lambda_2$. Then
$(G_{\alpha\overline{\beta}})$ is a positive definite Hermitian matrix if and only $\lambda_1>0$ and $\lambda_2>0$, which is equivalently to
\begin{equation}
2c_0+t\phi_s+s(t-s)\phi_{ss}>0\quad \mbox{and}\quad k_1>0.
\end{equation}
Note that since $\phi>0$ and
$$
\phi[2c_0+t\phi_s+s(t-s)\phi_{ss}]=\phi^2+s(t-s)\phi_s^2+k_1,
$$
which implies that if $k_1>0$ then it necessary that $2c_0+t\phi_s+s(t-s)\phi_{ss}>0$. This proves \eqref{de2}.

If $n\geq 3$, then the other eigenvalues of $(G_{\alpha\overline{\beta}})$ are given by $\lambda_3=\cdots=\lambda_n=c_0$.
Thus $(G_{\alpha\overline{\beta}})$ is a positive definite Hermitian matrix if and only the conditions \eqref{dg3} are satisfied. This completes the proof.
\end{proof}

Let $F(z,v)=\sqrt{r\phi(t, s)}$ be a $U(n)$-invariant strongly convex complex Finsler metric with $r, s, t$ given by \eqref{rst}. In the following, we shall derive the real fundamental tensor matrix associated to $F$. For this purpose, we denote
\begin{eqnarray*}
z^{\alpha}&=&x^{\alpha}+\sqrt{-1}x^{\alpha+n},\quad \mathcal{J}x^{\alpha}=x^{\alpha+n},\quad \mathcal{J}x^{\alpha+n}=-x^{\alpha},\quad\alpha=1,\cdots,n,\\
v^{\alpha}&=&u^{\alpha}+\sqrt{-1}u^{\alpha+n},\quad \mathcal{J}u^\alpha=u^{\alpha+n},\quad \mathcal{J}u^{\alpha+n}=-u^\alpha, \quad \alpha=1,\cdots,n.
\end{eqnarray*}
It is easy to check that
\begin{eqnarray*}
\langle z,v\rangle&=&\langle x|u\rangle+\sqrt{-1}\langle\mathcal{J}x|u\rangle=\langle x|u\rangle-\sqrt{-1}\langle x|\mathcal{J}u\rangle,\\
\langle x|\mathcal{J}x\rangle&=&\langle u|\mathcal{J}u\rangle=0,\quad \langle x|\mathcal{J}u\rangle=-\langle \mathcal{J}x|u\rangle,\quad \langle\mathcal{J}x|\mathcal{J}u\rangle=\langle x|u\rangle,
\end{eqnarray*}
where $\langle\cdot|\cdot\rangle$ denotes the real canonical Euclidean inner product on $\mathbb{R}^{2n}$. It follows that
\begin{equation*}
|\langle z,v\rangle|^2=\langle x|u\rangle^2+\langle Jx|u\rangle^2=\langle x|u\rangle^2+\langle x|\mathcal{J}u\rangle^2.
\end{equation*}
Furthermore, we have
\begin{align}
& r=\langle v,v\rangle=\langle u|u\rangle,\quad t=\langle z,z\rangle=\langle x|x\rangle,\label{c1}\\
& s=\frac{|\langle z,v\rangle|^2}{r}=\frac{\langle x|u\rangle^2+\langle x|\mathcal{J}u\rangle^2}{r}=\frac{\langle u|x\rangle^2+\langle u|\mathcal{J}x\rangle^2}{r}.\label{c2}
\end{align}

\begin{proposition} \label{p1}\ \ Let $r, s, t$ be given by \eqref{rst}. Then
\begin{eqnarray}
&& r_i=2u^i,\quad s_i=\frac{2}{r}\left[\langle u|x\rangle x^i+\langle u|\mathcal{J}x\rangle\mathcal{J}x^i-su^i\right],\quad t_{;i}=2x^i,\label{c3}\\
&& s_{;i}=\frac{2}{r}\big[\langle x|u\rangle u^i-\langle u|\mathcal{J}x\rangle\mathcal{J}u^i\big],\label{c4}\\
&& s_{ij}=\frac{2}{r}\big[x^ix^j+\mathcal{J}x^i\mathcal{J}x^j-s_iu^j-s_ju^i-s\delta_{ij}\big],\label{c5}\\
&& s_{i;j}=-\frac{4u^i}{r^2}\Big[\langle x|u\rangle u^j-\langle u|\mathcal{J}x\rangle\mathcal{J}u^j\Big]+\frac{2}{r}\Big[x^iu^j+\langle x|u\rangle\delta_j^i-\mathcal{J}x^i\mathcal{J}u^j-\langle u|\mathcal{J}x\rangle\frac{\partial \mathcal{J}u^j}{\partial u^i}\Big],\label{c6}\\
&& \sum_{i=1}^{2n}(s_i)^2=\frac{4}{r}s(t-s),\quad \sum_{i=1}^{2n}s_is_{;i}=0,\label{c7}\\
&& s_iu^i=0,\quad s_ix^i=\frac{2}{r}(t-s)\langle u|x\rangle,\quad s_i\mathcal{J}x^i=\frac{2}{r}(t-s)\langle u|\mathcal{J}x\rangle,\label{c8}\\
&& s_{;i}x^i=2s,\quad s_{;i}\mathcal{J}x^i=0,\quad s_{;i}u^i=2\langle x|u\rangle,\quad t_{;i}x^i=2t,\label{c9}\\
&& s_{i;j}u^j=2x^i-s_{;i}.\label{c10}
\end{eqnarray}
\end{proposition}
\begin{proof}
By \eqref{c1} and \eqref{c2}, we get \eqref{c3}-\eqref{c6}. Note that $r, t, s$ are all homogeneous functions of $x$ and $u$,  after a series of contractions and by using Euler's theorem on homogeneous functions, we obtain \eqref{c7}-\eqref{c10}.
\end{proof}

Differentiating $G=r\phi(t, s)$ with respect to $u^i$ and $u^j$ successively, we get
\begin{eqnarray}
G_i&=&2u^i\phi+r\phi_ss_i, \nonumber\\
G_{ij}&=&2\phi\delta_{ij}+2u^i\phi_ss_j+2u^j\phi_ss_i+r\phi_{ss}s_is_j+r\phi_ss_{ij}.\label{22}
\end{eqnarray}
Substituting \eqref{c5} into \eqref{22}, we have
$$
G_{ij}=2(\phi-s\phi_s)\delta_{ij}+r\phi_{ss}s_is_j+\phi_s(2x^ix^j+2\mathcal{J}x^i\mathcal{J}x^j).
$$

\begin{proposition}\label{p2} \ \
Let $F$ be a $U(n)$-invariant complex Finsler metric on a domain $D \subseteq \mathbb{C}^n$. Then the real fundamental tensor associated to  $F$ is given by
 \begin{equation*}
 g_{ij}=\frac{1}{2}G_{ij}=(\phi-s\phi_{s})\delta_{ij}+\frac{1}{2}r\phi_{ss}s_{i}s_{j}+\phi_{s}(x^ix^j+\mathcal{J}x^i\mathcal{J}x^j).
 \end{equation*}
\end{proposition}

In the following, we denote
\begin{eqnarray}
&& c_0:=\phi-s\phi_s, \quad H:=\left(g_{ij}\right)_{2n\times 2n},\nonumber\\
&&B=\begin{pmatrix}
                      s_1 & x^1& \mathcal{J}x^1 \\
                      \vdots & \vdots&\vdots \\
                      s_{2n} & x^{2n}&\mathcal{J}x^{2n} \\
                    \end{pmatrix}_{2n\times 3},
                    \quad X=\begin{pmatrix}
                                            \frac{1}{2}r\phi_{ss} & 0&0 \\
                                            0 & \phi_s &0\\
                                            0&0&\phi_s\\
                                          \end{pmatrix}.\label{bx}
\end{eqnarray}
We also denote $B^T$ the transpose of matrix $B$. Then by Proposition \ref{p1} and \ref{p2} we have
\begin{eqnarray}
H&=&c_0I_{2n}+BXB^T,\\
B^TB&=&\begin{pmatrix}
      \sum\limits_{i=1}^{2n}(s_i)^2 & s_ix^i &  s_i\mathcal{J}x^i \\
        s_ix^i &  \sum\limits_{i=1}^{2n}(x^i)^2 &  \sum\limits_{i=1}^{2n}x^i\mathcal{J}x^i \\
        s_i\mathcal{J}x^i &  \sum\limits_{i=1}^{2n}x^i\mathcal{J}x^i &  \sum\limits_{i=1}^{2n}(\mathcal{J}x^i)^2 \\
     \end{pmatrix}\nonumber\\
    &=&\begin{pmatrix}
      \frac{4}{r}s(t-s) & \frac{2}{r}(t-s)\langle u|x\rangle & \frac{2}{r}(t-s)\langle u|\mathcal{J}x\rangle\\
          & &\\
      \frac{2}{r}(t-s)\langle u|x\rangle & t & 0 \\
         &  &\\
       \frac{2}{r}(t-s)\langle u|\mathcal{J}x\rangle & 0 & t \\
     \end{pmatrix}.
     \label{btb}
\end{eqnarray}

\begin{proposition}\ \ \label{p5}
Let $F(z,v)=\sqrt{r\phi(t,s)}$ be a $U(n)$-invariant strongly convex complex Finsler metric on a domain $D \subseteq \mathbb{C}^n$. Then the inverse matrix $(g^{ij})$ of its real fundamental tensor matrix $(g_{ij})$ is given by
\begin{equation}
g^{jk}=\frac{1}{c_0}\bigg\{\delta_{jk}-\frac{r\phi_{ss}}{2L}X^jX^k-\frac{\phi_s}{c_0+t\phi_s}\left(x^jx^k+\mathcal{J}x^j\mathcal{J}x^k\right)\bigg\},\label{gi}
\end{equation}
where
\begin{eqnarray}
L&=&(c_0+t\phi_s)\big[c_0(c_0+t\phi_s)+2s(t-s)\phi\phi_{ss}\big],\label{l}\\
X^j&=&\phi s_j-\frac{2s(t-s)\phi_su^j}{r},\qquad j=1,\cdots,2n.
\end{eqnarray}
\end{proposition}

\begin{proof}
By Proposition \ref{p2}, the real fundamental tensor matrix $H=(g_{ij})$ is given by
$$
H=c_0I_{2n}+BXB^T.
$$
It is possible to obtain the inverse matrix of $H$ by setting
\begin{equation}
H^{-1}=\frac{1}{c_0}I_{2n}-BZB^T, \label{11}
\end{equation}
where $Z$ is a $3$-by-$3$ matrix to be determined. Since by direct calculation we have
 \begin{eqnarray*}
 H^{-1}H&=&I_{2n}-B\Big(ZB^TBX+c_0Z-\frac{1}{c_0}X\Big)B^T.
 \end{eqnarray*}
Thus in order to determine the $3$-by-$3$ matrix $Z$, it suffices to take
$$
ZB^TBX+c_0Z-\frac{1}{c_0}X=0,
$$
or equivalently
\begin{equation*}
Z=\frac{1}{c_0}X(c_0I_3+B^TBX)^{-1}
\end{equation*}
in the case that the $3$-by-$3$ matrix $c_0I_3+B^TBX$ is invertible. This
together with \eqref{bx} and \eqref{btb} gives
\begin{equation}
Z=\frac{1}{c_0L}\begin{pmatrix}
    c_{11} & c_{12} & c_{13} \\
    c_{21} & c_{22} & c_{23} \\
    c_{31} & c_{32} & c_{33} \\
  \end{pmatrix},\label{z}
\end{equation}
where
\begin{align*}
&L=\det(c_0I_3+B^TBX)=(c_0+t\phi_s)\big[c_0(c_0+t\phi_s)+2s(t-s)\phi\phi_{ss}\big],\\
&c_{11}=\frac{1}{2}r\phi_{ss}(c_0+t\phi_s)^2,\\
&c_{22}= \phi_s(c_0+t\phi_s)\big[c_0+2s(t-s)\phi_{ss}\big]-\frac{2}{r}\phi_s^2\phi_{ss}(t-s)^2\langle u|\mathcal{J}x\rangle^2,\\
&c_{33}= \phi_s(c_0+t\phi_s)\big[c_0+2s(t-s)\phi_{ss}\big]-\frac{2}{r}\phi_s^2\phi_{ss}(t-s)^2\langle u|x\rangle^2,\\
&c_{12}=c_{21}=-(t-s)\phi_s\phi_{ss}(c_0+t\phi_s)\langle u|x\rangle,\\
&c_{13}=c_{31}=-(t-s)\phi_s\phi_{ss}[c_0+t\phi_s]\langle u|\mathcal{J}x\rangle,\\
&c_{23}=c_{32}=\frac{2}{r}\phi_s^2\phi_{ss}(t-s)^2\langle u|x\rangle \langle u|\mathcal{J}x\rangle.
\end{align*}
Substituting \eqref{z} into \eqref{11}, we obtain
\begin{align}
 g^{jk}
 &=\frac{1}{c_0}\delta_{ij}-\frac{1}{c_0L}
   \begin{pmatrix}
     s_j & x^j & \mathcal{J}x^j \\
   \end{pmatrix}
 \begin{pmatrix}
    c_{11} & c_{12} & c_{13} \\
    c_{21} & c_{22} & c_{23} \\
    c_{31} & c_{32} & c_{33} \\
  \end{pmatrix}
  \begin{pmatrix}
                          s_k \\
                          x^k \\
                          \mathcal{J}x^k \\
                        \end{pmatrix}\nonumber\\
&=\frac{1}{c_0}\delta_{ij}-\frac{1}{c_0L}\bigg[
   c_{11}s_js_k+c_{12}\left(s_jx^k+s_kx^j\right)+c_{13}\left(s_j\mathcal{J}x^k+s_k\mathcal{J}x^j\right)\nonumber\\
&\quad +c_{22}x^jx^k+c_{23}\left(x^j\mathcal{J}x^k+x^k\mathcal{J}x^j\right)
   +c_{33}\mathcal{J}x^j\mathcal{J}x^k
\bigg]\nonumber\\
&=\frac{1}{c_0}\delta_{ij}-\frac{1}{c_0L}\bigg\{
   \frac{1}{2}r\phi_{ss}(c_0+t\phi_s)^2s_js_k\nonumber\\
&  \quad -(t-s)\phi_s\phi_{ss}(c_0+t\phi_s)\Big[s_j\left(\langle u|x\rangle x^k+\langle u| \mathcal{J}x\rangle\mathcal{J}x^k\right)+s_k\left(\langle u| x\rangle x^j+\langle u| \mathcal{J}x\rangle \mathcal{J}x^j\right)\Big]\nonumber\\
&\quad +\frac{2}{r}\phi_s^2\phi_{ss}(t-s)^2\Big[\langle u| x\rangle x^j\langle u| \mathcal{J}x\rangle \mathcal{J}x^k+\langle u| x\rangle x^k\langle u| \mathcal{J}x\rangle \mathcal{J}x^j\Big]\nonumber\\
&\quad +\phi_s(c_0+t\phi_s)\big[c_0+2s(t-s)\phi_{ss}\big]\left(x^jx^k+\mathcal{J}x^j\mathcal{J}x^k\right)\nonumber\\
&\quad -\frac{2}{r}\phi_s^2\phi_{ss}(t-s)^2\left[\langle u|\mathcal{J}x\rangle^2x^jx^k+\langle u| x\rangle^2\mathcal{J}x^j\mathcal{J}x^k\right]
\bigg\}.\label{g}
\end{align}
A rearrangement of \eqref{g} together with the equalities
$$
(c_0+t\phi_s)s_i-2(t-s)\phi_s\frac{\langle u|x\rangle x^i+\langle u|\mathcal{J}x\rangle \mathcal{J}x^i}{r}=\phi s_i-\frac{2s(t-s)\phi_su^i}{r}
$$
and
$$
\langle u|\mathcal{J}x\rangle^2=rs-\langle u| x\rangle^2
$$
gives \eqref{gi}. This completes the proof.
\end{proof}

\section{Strong convexity of $U(n)$-invariant complex Finsler metric}\label{sec4}

In \cite{Zh1}, the third author initiated a study on $U(n)$-invariant complex Finsler metrics and obtained the necessary and sufficient condition for a $U(n)$-invariant complex Finsler metrics to be strongly pseudoconvex, as well as a sufficient condition for a $U(n)$-invariant complex Finsler metric to be strongly convex. In this section, we shall give a necessary and sufficient condition for a $U(n)$-invariant complex Finsler metric to be strongly convex.
\begin{proposition}[\cite{Zh1}]\label{p3}\ \
A $U(n)$-invariant complex Finsler metric $F(z,v)=\sqrt{r\phi(t,s)}$ defined on a domain $D \subseteq \mathbb{C}^n $  is  strongly pseudoconvex
 if and only if $\phi$ satisfies
 \begin{equation*}
  \phi-s\phi_s>0,\qquad
(\phi-s\phi_s)\big[\phi+(t-s)\phi_s\big]+s(t-s)\phi\phi_{ss} > 0
 \end{equation*}
whenever $n\geq 3$; or
\begin{equation*}
(\phi-s\phi_s)\big[\phi+(t-s)\phi_s\big]+s(t-s)\phi\phi_{ss} > 0
\end{equation*}
whenever $n=2$.
\end{proposition}

\begin{proposition}[\cite{Zh1}]\label{p4}\ \
Let $F=\sqrt{r\phi(t,s)}$ be a $U(n)$-invariant complex Finsler metric on a domain $D \subseteq \mathbb{C}^n$. If $F$ satisfies
\begin{equation*}
\phi-s\phi_s>0, \quad \phi_s\geq 0,\quad \phi_{ss}\geq 0
\end{equation*}
for every $z\in D$ and nonzero vector $v\in T_z^{1,0}D$. Then $F$ is a strongly convex complex Finsler metric.
\end{proposition}

Denote $M_{n\times m}(\mathbb{C})$ the set of all $n \times m$ matrices over the complex number field $\mathbb{C}$. We need the following lemma.
\begin{lemma}[\cite{XZ3}]\label{le1}\ \
Let $C\in M_{n\times m}(\mathbb{C}), E\in M_{m\times n}(\mathbb{C}), \lambda\in \mathbb{C}$. Then
\begin{equation*}
\lambda^m\det(\lambda I_n-CE)=\lambda^{n}\det(\lambda I_{m}-EC).
\end{equation*}
\end{lemma}

\begin{theorem}\label{iffsc} \
A $U(n)$-invariant complex Finsler metric $F(z,v)=\sqrt{r\phi(t,s)}$ defined on a domain $D \subseteq \mathbb{C}^n$  is  strongly convex
 if and only if
 \begin{eqnarray}
 &&\phi-s\phi_s>0,\label{s1}\\
 &&\phi+(t-s)\phi_s>0,\label{s2}\\
 &&(\phi-s\phi_s)[\phi+(t-s)\phi_s]+2s(t-s)\phi \phi_{ss}> 0\label{s3}
 \end{eqnarray}
for any $z\in D$ and any nonzero vector $v\in T_z^{1,0}D$.
\end{theorem}

\begin{proof}
We prove this theorem by derive the necessary and sufficient condition that all of the eigenvalues of $H$ are positive.

By Proposition \ref{p2}, the fundamental tensor matrix $H$ is
$$
H=c_0I_{2n}+BXB^T.
$$
Using \eqref{btb}, we get
\begin{equation*}
B^TBX
 =\left(
      \begin{array}{ccc}
        2s(t-s)\phi_{ss}&  \frac{2}{r}\phi_s(t-s)\langle u|x\rangle& \frac{2}{r}\phi_s(t-s)\langle u|\mathcal{J}x\rangle \\
        \phi_{ss}(t-s)\langle u|x\rangle & t\phi_s&0 \\
        \phi_{ss}(t-s)\langle u|\mathcal{J}x\rangle&0&t\phi_s
      \end{array}
    \right).
\end{equation*}
In the following, we derive all of the eigenvalues of the fundamental tensor matrix $H$. By Lemma \ref{le1}, we have
\begin{eqnarray*}
&&\det(\lambda I_{2n} - H)\\
&=&\det\left[(\lambda- c_0)I_{2n} - BXB^T\right]\\
&=& (\lambda- c_0)^{2n-3}\det\left[(\lambda- c_0)I_3 - B^TBX\right]\\
&=&(\lambda- c_0)^{2n-3}
    \det\left(
      \begin{array}{ccc}
       (\lambda - c_0)-2s(t-s)\phi_{ss} &  -\frac{2}{r}\phi_s(t-s)\langle u|x\rangle &-\frac{2}{r}\phi_s(t-s)\langle u|\mathcal{J}x\rangle \\
        -\phi_{ss}(t-s)\langle u|x\rangle & (\lambda-c_0)-t\phi_s&0 \\
        -\phi_{ss}(t-s)\langle u|\mathcal{J}x\rangle&0&(\lambda-c_0)-t\phi_s
      \end{array}
    \right)\\
&=&(\lambda-c_0)^{2n-3}(\lambda-c_0-t\phi_s)\Big\{\lambda^2-\big[2c_0+t\phi_s+2s(t-s)\phi_{ss}\big]\lambda+\tilde{k}\Big\},
 \end{eqnarray*}
where
\begin{equation*}
\tilde{k}=c_0\left(c_0+t\phi_s\right)+2s(t-s)\phi\phi_{ss}.
\end{equation*}

Note that an $2n$-by-$2n$ real symmetric matrix $H$ has exactly $2n$ real eigenvalues(counting multiplicities). Therefore
$F$ is a strongly convex complex Finsler metric if and only if $H$ is a positive definite matrix at each $(z,v)\in T^{1,0}D\cong D\times \mathbb{C}^n$ with $v\neq0$,
if and only if all the eigenvalues of $H$ is positive, if and only if
\begin{equation}
c_0>0, \quad c_0+t\phi_s> 0, \quad 2c_0+t\phi_s+2s(t-s)\phi_{ss}>0,\quad \tilde{k}>0.\label{00}
\end{equation}

It is easy to check that
$$
\phi\big[2c_0+t\phi_s+2s(t-s)\phi_{ss}\big]=\phi^2+s(t-s)\phi_s^2+\tilde{k}.
$$
Thus if the fourth inequality in \eqref{00} holds, then it is necessary that the third inequality in \eqref{00} holds since $\phi>0$. Therefore $F$ is a strongly convex complex Finsler metric if and only if \eqref{s1}-\eqref{s3} hold.
This completes the proof.
\end{proof}

\begin{remark}\label{rm}
There are lots of functions $\phi(t,s)$ which satisfy Theorem \ref{iffsc}. As a result, there are lots of $U(n)$-invariant strongly convex complex Finsler metrics on domains $D\subseteq\mathbb{C}^n$.
\end{remark}

\begin{example} Taking $\phi(t,s)=(1+s)^2$, then it is easy to check that
 $F(z,v)=\sqrt{r\phi(t,s)}$ is a $U(n)$-invariant strongly convex  complex Finsler metric on the open unit ball $\mathbb{B}^n\subset\mathbb{C}^n$.
\end{example}

Notice that a strongly convex complex Finsler metric is necessary a strongly pseudoconvex complex Finsler metric. The converse, however, is not necessary true, unless some extra condition is satisfied.

\begin{corollary}\ \
Let $F=\sqrt{r\phi(t,s)}$ be a $U(n)$-invariant complex Finsler metric on a domain $D \subseteq \mathbb{C}^n (n\geq 3)$. If $F$ satisfies $\phi_{ss}\geq 0$. Then $F$ is strongly pseudoconvex if and only if $F$ is strongly convex.
\end{corollary}
\begin{proof}
We only need to prove the necessity. Suppose that $F=\sqrt{r\phi(t,s)}$ is a $U(n)$-invariant strongly pseudoconvex complex Finsler metric satisfying $\phi_{ss}>0$. By Proposition \ref{p3}, $F$ satisfies
\begin{equation}
c_0=\phi-s\phi_s>0,\quad
k_1=c_0(c_0+t\phi_s)+s(t-s)\phi\phi_{ss} > 0.\label{ck}
\end{equation}
To prove the strong convexity of $F$, we only need to show that
$$
c_0+t\phi_s>0, \quad \tilde{k}>0.
$$
In deed, we have
\begin{align*}
\tilde{k}=k_1+s(t-s)\phi\phi_{ss},\qquad \frac{k_1}{\phi^2}&=\left[\frac{s(c_0+t\phi_s)}{\phi}\right]_s.
\end{align*}
It follows from \eqref{ck} that $\tilde{k}>0$ since $\phi_{ss}\geq 0$, and that $\frac{s(c_0+t\phi_s)}{\phi}$ is strictly monotonically increasing with respect to $s$. Hence $\frac{s(c_0+t\phi_s)}{\phi}>0$ whenever $s>0$, which implies that $c_0+t\phi_s>0$ whenever $s>0$. When $s=0$, we can deduce directly from $k_1>0$ that $c_0+t\phi_s>0$. Hence $c_0+t\phi_s>0$ for any $s\geq 0$ and $t\geq 0$. This completes the proof.
\end{proof}

In the following, we denote $\mathbb{B}^n(R)=\{z\in \mathbb{C}^n: \|z\|^2<R^2\}$ the ball centered at the origin with radius $R$ in $\mathbb{C}^n$. Especially, we denote $\mathbb{B}^n=\mathbb{B}^n(1)$ the  open unit ball in $\mathbb{C}^n$.
Now we give an example of strongly pseudoconvex complex Finsler metric which is not strongly convex.
\begin{example}\label{e2}\ \
Let
$$
\phi(t, s)=4-s^2.
$$
Then $F=\sqrt{r\phi(t,s)}$ is a  $U(n)$-invariant strongly pseudoconvex complex Finsler metric on the ball $\mathbb{B}^n(\sqrt[4]{3})\subset\mathbb{C}^n$, but $F$ is not a strongly convex complex Finsler metric.
\end{example}

\begin{proof}
It is obvious that
$$
0\leq s\leq t<\sqrt{3}.
$$
In the following, we divide our proof into two steps for clarity.

{\bf Step 1}.  We show that $F$ is strongly pseudoconvex. A direct calculation gives
\begin{align}
\phi_s&=-2s, \quad \phi_{ss}=-2,\nonumber\\
c_0&=\phi-s\phi_s=4+s^2>0,\label{c0}\\
k_1&=c_0(c_0+t\phi_s)+s(t-s)\phi\phi_{ss}\nonumber\\
 &=16-s^4+16s^2-16st\nonumber\\
 &>16-s^4+16s^2-16\sqrt{3}s,  \label{f}
\end{align}
where we have used the fact that $t<\sqrt{3}$ in the last inequalitiy. Denote
\begin{equation}
f(s)=16-s^4+16s^2-16\sqrt{3}s. \label{fs}
\end{equation}

Nest we prove that $F$ is strong pseudoconvexity. By Proposition \ref{p3} we only need to show that $f(s)\geq 0$ for $0<s<\sqrt{3}$. Note that $f(0)=16>0$, $f(\sqrt{3})=7>0$, It suffices to show $f(s)\geq 0$ at extremal points. Denote $s$ one of the extremal point if it causes no confusion. Then at any extremal point,  we have
$$
f'(s)=-4s^3+32s-16\sqrt{3}=0,
$$
which yields that any extremal point of $f(s)$ satisfies
\begin{equation}
8s-s^3=4\sqrt{3}. \label{ss}
\end{equation}
Substituting \eqref{ss} into \eqref{fs}, we get
\begin{align*}
f(s)&=16-s^4+16s^2-16\sqrt{3}s\\
    &=16+(8s-s^3)s+8s^2-16\sqrt{3}s\\
    &=16+4\sqrt{3}s+8s^2-16\sqrt{3}s\\
    &=(2\sqrt{2}s-3)^2+7\\
    &\geq 7 > 0.
\end{align*}
Hence $f(s)>0$ for any $s$, furthermore by \eqref{f} we get $k_1>0$ for any $s$ and $t$. This together with \eqref{c0} shows that $F$ is a strongly pseudoconvex complex Finsler metric.

{\bf Step 2}.  We prove that $F$ is not strongly convex. It suffices to show at some points, $\tilde{k}\leq 0$. In fact, we have
\begin{align}
\tilde{k}
  &=c_0(c_0+t\phi_s)+2s(t-s)\phi\phi_{ss}\nonumber\\
  &=(4+s^2)(4+s^2-2ts)-4s(t-s)(4-s^2).\label{tk}
\end{align}
By taking $s=\frac{t}{2}$ in \eqref{tk}, we get
\begin{align}
\tilde{k}
  &=\left(4+\frac{t^2}{4}\right)\left(4+\frac{t^2}{4}-t^2\right)-4\times\frac{t}{2}\left(t-\frac{t}{2}\right)\left(4-\frac{t^2}{4}\right)\nonumber\\
  &=16+\frac{1}{16}t^4-6t^2.\label{tk1}
\end{align}
Since $t^2\in[0, 3)$, we can take $t^2=3-\varepsilon$ in \eqref{tk1} for some positive real number $\varepsilon$ small enough, in this case
\begin{align*}
\tilde{k}
  &=16+\frac{1}{16}(3-\varepsilon)^2-6(3-\varepsilon)\\
  &=-\frac{23}{16}+O(\varepsilon)\\
  &<0,
\end{align*}
where $O(\varepsilon)$ is an infinitesimal of first order of $\varepsilon$. This completes the proof.
\end{proof}

\section{$U(n)$-invariant strongly convex real Berwald metric}\label{sec5}

By Remark \ref{rm}, there are lots of $U(n)$-invariant strongly convex complex Finsler metrics on domains $D\subseteq\mathbb{C}^n$, one may ask that among these metrics whether there are real Berwald metrics.
In this section, we shall give a characterization of $U(n)$-invariant strongly convex complex Finsler metrics to be real Berwald metrics.
\begin{proposition}\ \ \label{p6}
Let $F=\sqrt{r\phi(t,s)}$ be a $U(n)$-invariant strongly convex complex Finsler metric on a domain $D \subseteq \mathbb{C}^n$. Then its real spray coefficients $\textbf{G}^i$ are given by
\begin{equation}
\textbf{G}^i=\left[c_1\langle x|u\rangle^2+rc_2\right]x^i+c_1\langle x|u\rangle\langle u|\mathcal{J}x\rangle\mathcal{J}x^i+c_3\langle x|u\rangle u^i+c_4\langle u|\mathcal{J}x\rangle\mathcal{J}u^i,\label{spray}
\end{equation}
where
\begin{eqnarray}
c_1&=&\frac{1}{Lc_0}\Big\{c_0(c_0+t\phi_s)\big[\phi(\phi_{st}+\phi_{ss})-\phi_s(\phi_t+\phi_s)\big]\nonumber\\
    &&\qquad\quad -(t-s)\phi\phi_{ss}\big[s\phi_s(\phi_t+\phi_s)-\phi(\phi_t-\phi_s)\big]\Big\},\label{t1}\\
c_2&=&\frac{s\phi_s(\phi_t+\phi_s)-\phi(\phi_t-\phi_s)}{2c_0(c_0+t\phi_s)},\label{t2}\\
c_3&=&\frac{(c_0+t\phi_s)(\phi_t-s\phi_{st})+s\phi_{ss}[(t-s)\phi_t-\phi]}{c_0(c_0+t\phi_s)+2s(t-s)\phi\phi_{ss}},\label{t3}\\
c_4&=&\frac{\phi_s}{c_0}.\label{t4}
\end{eqnarray}
\end{proposition}
\begin{proof}
The real spray coefficients associated to $F$ are given by
\begin{equation}
\textbf{G}^i=\frac{1}{4}g^{il}\left(G_{l;k}u^k-G_{;l}\right).\label{spr}
\end{equation}
By Proposition \ref{p5}, we have
\begin{equation}
g^{il}=\frac{1}{c_0}\bigg\{\delta_{il}-\frac{r\phi_{ss}}{2L}X^iX^l-\frac{\phi_s}{c_0+t\phi_s}\left(x^ix^l+\mathcal{J}x^i\mathcal{J}x^l\right)\bigg\},\label{gil}
\end{equation}
where $L$ is given by \eqref{l}, and
\begin{equation}
X^i=\phi s_i-\frac{2}{r}s(t-s)\phi_su^i=\frac{2}{r}\phi\Big[\langle x|u\rangle x^i+\langle u|\mathcal{J}x\rangle\mathcal{J}x^i\Big]-\frac{2}{r}s(c_0+t\phi_s)u^i. \label{G0}
\end{equation}
 Since $G=r\phi(t, s)$, a direct calculation yields
\begin{equation}
G_{l;k}u^k-G_{;l}=2\langle x|u\rangle \left[2(\phi_t+\phi_s)u^l+r(\phi_{st}+\phi_{ss})s_l\right]+2r(\phi_s-\phi_t)x^l-2r\phi_ss_{;l}.\label{G}
\end{equation}
By substituting $\eqref{c3}$ and \eqref{c4} into \eqref{G}, we get
\begin{align}
G_{l;k}u^k-G_{;l}
&=4\big[\phi_t-s(\phi_{st}+\phi_{ss})\big]\langle x|u\rangle u^l+4\phi_s\langle u| \mathcal{J}x\rangle \mathcal{J}u^l+2(\phi_s-\phi_t)\langle u|u\rangle x^l\nonumber\\
 &\quad +4(\phi_{st}+\phi_{ss})\big[ \langle x|u\rangle^2x^l+\langle x|u\rangle \langle u|\mathcal{J}x\rangle\mathcal{J}x^l\big].                \label{G1}
\end{align}
By Proposition \ref{p1}, it is easy to check that
\begin{eqnarray*}
\sum\limits^{2n}_{l=1}X^lu^l&=&-2s(t-s)\phi_s,\\
\sum\limits_{l=1}^{2n}X^ls_l&=&\frac{4}{r}s(t-s)\phi, \\
\sum\limits^{2n}_{l=1}X^lx^l&=&\frac{2}{r}(t-s)c_0\langle x| u\rangle , \\
\sum\limits^{2n}_{l=1}X^ls_{;l}&=&-\frac{4}{r}s(t-s)\phi_s\langle x| u\rangle,
\end{eqnarray*}
which together with \eqref{G} imply
\begin{equation}
X^l\left(G_{l;k}u^k-G_{;l}\right)=4(t-s)\big[c_0\phi_s-(\phi+s\phi_s)\phi_t+2s\phi(\phi_{st}+\phi_{ss})\big]\langle x| u\rangle.\label{G2}
\end{equation}
Similarly, we have
\begin{eqnarray}
x^l\left(G_{l;k}u^k-G_{;l}\right)&=&4\big[(\phi_t+\phi_s)+(t-s)(\phi_{st}+\phi_{ss})\big]\langle x| u\rangle^2+2r\big[t(\phi_s-\phi_t)-2s\phi_s\big],\label{G3}\\
Jx^l\left(G_{l;k}u^k-G_{;l}\right)&=&4\big[(\phi_t+\phi_s)+(t-s)(\phi_{st}+\phi_{ss})\big]\langle x|u\rangle\langle u|\mathcal{J}x\rangle.\label{G4}
\end{eqnarray}
By substituting \eqref{gil}, \eqref{G0}, \eqref{G1}-\eqref{G4} into \eqref{spr}, and rearranging the terms according to different types, we obtain \eqref{spray}.
\end{proof}

\begin{theorem}\ \
Let $F=\sqrt{r\phi(t,s)}$ be a $U(n)$-invariant strongly convex complex Finsler metric on a domain $D \subset \mathbb{C}^n$. Then $F$ is a real Berwald metric if and only if $F$ comes from a $U(n)$-invariant Hermitian metric.
\end{theorem}

\begin{proof}
If $F$ comes from a $U(n)$-invariant Hermitian metric we have $\phi_{ss}=0$ and consequently
$$
g_{ij}=c_0\delta_{ij}+\phi_s(x^ix^j+\mathcal{J}x^i\mathcal{J}x^j).
$$
It is clear that
$$
\frac{\partial c_0}{\partial s}=-s\phi_{ss}=0,\quad\frac{\partial \phi_s}{\partial s}=\phi_{ss}=0.
$$
Thus $F$ is a real Berwald metric. Conversely, If $F$ is a real Berwald metric, then the real spray coefficients $\textbf{G}^i$ of $F$ are quadratic with respect to the tangent directions $u$. If follows from \eqref{spray} that $c_1$, $c_3$ and $c_4$ are necessary independent of $s$, i.e.,
\begin{equation}
 \frac{ \partial c_1}{\partial s}=0, \quad \frac{ \partial c_3}{\partial s}=0,\quad \frac{ \partial c_4}{\partial s}=0.\label{cc}
\end{equation}
Next since
$$
rs=\langle x| u\rangle^2+\langle u| \mathcal{J}x\rangle^2
$$
are quadratic with respect to $u$, thus $c_2$ is at most linear with respect to $s$, i.e.,
\begin{equation}
 \frac{ \partial^2 c_2}{\partial s^2}=0.\label{cc1}
\end{equation}

But the equality
$$
0=\frac{\partial c_4}{\partial s}=\frac{\phi\phi_{ss}}{c_0^2}
$$
implies
$\phi_{ss}=0$, i.e., $F$ comes from a $U(n)$-invariant Hermitian metric. In this case, equality \eqref{cc1} and the rest three equalities in \eqref{cc}  hold automatically. This completes the proof.
\end{proof}

\section{Holomorphic sectional curvature of $U(n)$-invariant complex Finsler metrics}\label{sec6}

A strongly pseudoconvex complex Finsler metric is called a weakly complex Berwald metric if its complex Berwald connection coefficients $\mathbb{G}_{\beta\gamma}^\alpha(z,v)=\mathbb{G}_{\beta\gamma}^\alpha(z)$, that is, locally $\mathbb{G}_{\beta\gamma}^\alpha$ are independent of fiber coordinates. This notion was first introduced in \cite{Zh2}. It is easy to check that a strongly pseudoconvex complex Finsler metric is a weakly complex Berwald metric if and only if locally its complex geodesic coefficients $\mathbb{G}^i$ are quadratic and holomorphic with respect to fiber coordinates $v$, that is,
$$\mathbb{G}_{\beta\gamma}^\alpha=\mathbb{G}_{\beta\gamma}^\alpha(z)v^\beta v^\gamma.$$
So that weakly complex Berwald metrics are those complex Finsler metrics whose behavior of geodesic are the most closest to Hermitian metrics.

Let $F=\sqrt{r\phi(t,s)}$ be a $U(n)$-invariant strongly pseudoconvex complex Finsler metric defined on a domain $D\subseteq\mathbb{C}^n$. It was proved in \cite{Zh1} that
 $F$ is a weakly complex Berwald metric if and only if $\phi$ satisfies the following nonlinear PDE
$$
\phi(\phi_{st}+\phi_{ss})-\phi_s(\phi_t+\phi_s)=g(t)\Big[c_0(c_0+t\phi_s)+s(t-s)\phi\phi_{ss}\Big]
$$
for a real-valued smooth function $g(t)$. For $U(n)$-invariant weakly complex Berwald metrics, its holomorphic sectional curvature along a nonzero tangent direction $v\in T_z^{1,0}D$ was given in \cite{Zh1},
$$
K_F(z,v)=-\frac{2}{\phi^2}\bigg\{\phi\Big[s\frac{\partial k_2}{\partial t}+k_2\Big]+s(c_0+t\phi_s)\big[sg'(t)+2g(t)\big]\bigg\},
$$
where
\begin{eqnarray*}
k_2&=&\frac{1}{k_1}\Big\{[\phi+(t-s)\phi_s+s(t-s)\phi_{ss}](\phi_t+\phi_s)-s[\phi+(t-s)\phi_s](\phi_{st}+\phi_{ss})\Big\},\\
\frac{\partial k_2}{\partial t}
&=&\frac{1}{\phi^2}\Big\{\phi(\phi_{tt}+\phi_{st})-\phi_t(\phi_t+\phi_s)-sg(t)\big[\phi(\phi_t+\phi_s)+(t-s)\phi\phi_{st}-\phi_t(c_0+t\phi_s)\big]\\
       && -s\phi g'(t)\big[\phi+(t-s)\phi_s\big]\Big\}.
\end{eqnarray*}

\begin{theorem}\label{vc}
Suppose that $F(z,v)=\sqrt{r\phi(t,s)}$ is a $U(n)$-invariant strongly pseudoconvex complex Finsler metric on a domain $D\subseteq\mathbb{C}^n$. Then  $F$ is a weakly complex Berwald metric with vanishing holomorphic sectional curvature if and only if $F=\sqrt{rf(s-t)}$ for some smooth positive function $f(w)$ with $w=s-t$.
\end{theorem}
\begin{proof}
The sufficiency follows from  Theorem 4.1 in \cite{XZ2} and the necessity follows from Theorem 5.8 in \cite{XZ1}.
\end{proof}

For an arbitrary $U(n)$-invariant strongly pseudoconvex complex Finsler metric $F=\sqrt{r\phi(t,s)}$, we obtain the following explicit formula for its holomorphic sectional curvature.
\begin{theorem}\ \ \label{th2}
Let $F(z,v)=\sqrt{r\phi(t,s)}$ be a $U(n)$-invariant  strongly pseudoconvex  complex Finsler metric. Then the holomorphic sectional curvature $K_F(z,v)$ of $F$ along a nonzero tangent direction $v\in T_z^{1,0}D$ is given by
\begin{align}
K_F(z,v)&=-\frac{2}{\phi^2k_1}\bigg\{k_1\Big[s(\phi_{tt}+2\phi_{st}+\phi_{ss})+(\phi_t+\phi_s)\Big]-s^2(t-s)\phi(\phi_{st}+\phi_{ss})^2\nonumber\\
 &\quad+2s^2(t-s)\phi_s(\phi_{st}+\phi_{ss})(\phi_t+\phi_s)-s\Big[c_0+(t-s)\phi_s+s(t-s)\phi_{ss}\Big](\phi_t+\phi_s)^2\bigg\},\label{kf}
\end{align}
where $c_0$ and $k_1$ are given by \eqref{ck}.
Especially, if $0\in D$, then at the origin the holomorphic sectional curvature of $F$ is always a constant along any nonzero tangent direction, i.e.,
 $$
 K_F(0,v)=\frac{-2[\phi_t(0,0)+\phi_s(0,0)]}{\phi^2(0,0)}=\mbox{constant},\quad \forall 0 \neq  v\in T_0^{1,0}D.
 $$
\end{theorem}

\begin{proof}
The holomorphic sectional curvature $K_F$ of a strongly pseudoconvex complex Finsler metric $F$ can be expressed as (see \cite{Zh1})
\begin{equation}
K_F=-\frac{2}{G^2}G_\gamma \partial_{\overline{\nu}}(2\mathbb{G}^\gamma)\overline{v^\nu},\label{h1}
\end{equation}
where
$$G_\gamma=\frac{\partial G}{\partial v^\gamma},\qquad \partial_{\overline{\nu}}=\frac{\partial}{\partial\overline{v^\nu}},
$$
and $\mathbb{G}^\gamma$ denotes the complex geodesic spray coefficients associated to $F$.
By (3.14) and (3.15) in \cite{Zh1}, we have
$$
2\mathbb{G}^\gamma=k_2\overline{\langle z,v\rangle}v^\gamma+k_3(\overline{\langle z,v\rangle})^2z^\gamma, \qquad \gamma=1, 2, \cdots, n.
$$
where
\begin{align}
k_2&=\frac{k_4}{k_1},\quad\qquad k_3=\frac{k_5}{k_1},\label{k2k3}\\
k_4&=\big[c_0+t\phi_s+s(t-s)\phi_{ss}\big](\phi_t+\phi_s)-s(c_0+t\phi_s)(\phi_{st}+\phi_{ss}),\label{k4}\\
k_5&=\phi(\phi_{st}+\phi_{ss})-\phi_s(\phi_t+\phi_s).\label{k5}
\end{align}
It is easy to check that
\begin{eqnarray}
\partial_{\overline{\nu}}(2\mathbb{G}^\gamma)&=&\frac{\partial k_2}{\partial t}z^\nu\overline{\langle z,v\rangle}v^\gamma+\frac{\partial k_2}{\partial s}\frac{\langle z,v\rangle v^\nu}{r}\overline{\langle z,v\rangle}v^\gamma+k_2v^\nu v^\gamma\nonumber\\
& &+\frac{\partial k_3}{\partial t}z^\nu(\overline{\langle z,v\rangle})^2z^\gamma +\frac{\partial k_3}{\partial s}\frac{\langle z,v\rangle}{r}v^\nu(\overline{\langle z,v\rangle})^2z^\gamma+2k_3\overline{\langle z,v\rangle}v^\nu z^\gamma.\label{gr}
\end{eqnarray}
Note that
\begin{align*}
&\sum\limits^n_{\nu=1}z^\nu\overline{v^\nu}=\langle z,v\rangle,\quad \sum\limits^n_{\nu=1}v^\nu\overline{v^\nu}=r,\quad \langle z,v\rangle\overline{\langle z,v\rangle}=rs,\\
&G_\gamma v^\gamma=G,\qquad G_\gamma z^\gamma=(\overline{v^\gamma}\phi+r\phi_ss_\gamma)z^\gamma=(c_0+t\phi_s)\langle z,v\rangle.
\end{align*}
Thus contracting \eqref{gr} with $\overline{v^\nu}$ and $G_\gamma$ successively yields
\begin{equation}
G_\gamma \partial_{\overline{\nu}}(2\mathbb{G}^\gamma)\overline{v^\nu}=r^2\bigg[\Big(\frac{\partial k_2}{\partial t}s+\frac{\partial k_2}{\partial s}s+k_2 \Big)\phi+s\Big(\frac{\partial k_3}{\partial t}s+\frac{\partial k_3}{\partial s}s+2k_3\Big)(c_0+t\phi_s)\bigg].\label{grr}
\end{equation}
Substituting \eqref{grr} and \eqref{k2k3}-\eqref{k5} into \eqref{h1} and rearranging terms, we obtain \eqref{kf}.

Furthermore, if $0\in D$, then at the origin $z=0$, we have $t=s=0$, which implies $k_1=\phi^2(0,0)$. Thus
$$
K_F(0,v)=-\frac{2}{\phi^2}(\phi_t+\phi_s)\Big|_{(t,s)=(0,0)}=\frac{-2[\phi_t(0,0)+\phi_s(0,0)]}{\phi^2(0,0)}=\mbox{constant}.
$$
This completes the proof.
\end{proof}

Hopefully, Theorem \ref{th2} shall find its independent interesting in looking for examples of $U(n)$-invariant complex Finsler metrics with specific property of holomorphic sectional curvatures.

\begin{example}\ \
Let $F=\sqrt{r\phi(t,s)}$ be a $U(n)$-invariant complex Finsler metric with $\phi(t,s)=\frac{(1-t+s)^2}{(1-t)^3}$. Then $F$ is a strongly convex weakly complex Berwald metric  on the open unit ball $\mathbb{B}^n\subset\mathbb{C}^n$ satisfying
\begin{eqnarray*}
K_F(z,v)&=&\frac{-6(1-t)}{1-t+s}\geq -6,\quad \forall z\in \mathbb{B}^n,0\neq v\in T_z^{1,0}\mathbb{B}^n,\\
K_F(0,v)&\equiv& -6,\quad \forall 0\neq v\in T_0^{1,0}\mathbb{B}^n.
\end{eqnarray*}
\end{example}

\begin{example}\cite{XZ2}
Let $F=\sqrt{rf(s-t)}$ be a $U(n)$-invariant strongly pseudoconvex complex Finsler metric defined on a domain $D\subseteq \mathbb{C}^n$. Then the holomorphic curvature of $F$ vanishes identically, i.e.,
$$
K_F(z,v)\equiv 0,\quad \forall z\in D, 0\neq v\in T_z^{1,0}D.
$$
\end{example}

\begin{example}\cite{XZ3}
Let $F=\sqrt{r\phi(t,s)}$ be a $U(n)$-invariant complex Finsler metric with $\phi(t,s)=\frac{(1+t-s)^2}{(1+t)^3}$. Then $F$ is a strongly convex weakly complex Berwald metric on the open unit ball $\mathbb{B}^n\subset\mathbb{C}^n$ satisfying
\begin{eqnarray*}
K_F(z,v)&=&\frac{6(1+t)}{1+t-s}\geq 6,\quad z\in \mathbb{B}^n, 0\neq v\in T_z^{1,0}\mathbb{B}^n,\\
 K_F(0,v)&\equiv& 6,\quad\forall 0\neq v\in T_0^{1,0}\mathbb{B}^n.
 \end{eqnarray*}
\end{example}

\section{The real geodesics on the unit sphere $\pmb{S}^{2n-1}\subset\mathbb{C}^n$}\label{sec7}

%In this section we consider $D=\mathbb{C}^n-\{0\}, n\geq 2$. In this case, the holomorphic tangent bundle $T^{1,0}D\cong D\times \mathbb{C}^n$. Let
%$$\varOmega=\Big\{(z,v)\in T^{1,0}D\,|\, z\neq \lambda v,\forall\lambda\in\mathbb{C},\lambda\neq 0\Big\}$$
% and define $F: \varOmega\rightarrow [0,+\infty)$ by
% \begin{equation}
% F(z,v)=\frac{|PQ|^2}{|OH|^2},\label{cw}
%  \end{equation}
%  where $P=(z^1,\cdots,z^n),Q=(z^1+v^1,\cdots,z^n+v^n)$, $OH\bot PQ, H\in PQ$, and $|PQ|$ denotes the Euclidean distance in $\mathbb{C}^n\cong\mathbb{R}^{2n}$. The complex Finsler metric $F(z,v)$ defined by \eqref{cw} is called the complex Wrona metric in \cite{SR}. It is easy to check that
%\begin{equation}
%F(z,v)=\frac{\|v\|^4}{\|z\|^2\|v\|^2-|\langle z,v\rangle|^2},\quad \forall (z,v)\in \varOmega.\label{wm}
%\end{equation}
%In our notation,
%$$F(z,v)=r\phi(t,s),\quad \phi(t,s)=\frac{1}{t-s},\quad t\neq s\quad\mbox{and}\quad t\neq 0,$$
%that is, it is a $U(n)$-invariant complex Finsler metric.

%In \cite{Zh2}, it was proved that the holomorphic curvature of the complex Wrona metric vanishes identically.

In \cite{SR}, S. Dragomir and R. Grimaldi posed an open problem to study the real (or complex geodesic) geodesics of the following complex Wrona metric
$$F(z,v)=\frac{\|v\|^4}{\|z\|^2\|v\|^2-|\langle z,v\rangle|^2}$$
on $\mathbb{C}^n$. For a parameterized curve $\gamma:[a,b]\rightarrow \mathbb{C}^n$ with $\tau_i=a+\frac{b-a}{m}i,i=0,1,\cdots,m-1$, they defined the length of $\gamma$
with respect to the complex Wrona metric by
$$
L(\gamma)=\lim_{m\rightarrow \infty}\sum_{i=0}^{m-1}F(\gamma(\tau_i),\;\gamma(\tau_{i+1})-\gamma(\tau_i)).
$$

In this section we consider the real geodesic of $U(n)$-invariant complex Finsler metrics on $\mathbb{C}^n$.

 S. Dragomir and R. Grimaldi \cite{SR} obtained the following interesting property of the complex Wrona metric.
\begin{proposition}[\cite{SR}]\ \
Let $0<\alpha<\frac{\pi}{2}$ and $\gamma(\tau)(0\leq \tau\leq\alpha)$ be any real geodesic of the complex Wrona metric on the unit sphere $\pmb{S}^{2n-1}\subset\mathbb{C}^n$, which is parametrized by arc length. Then
$L(\gamma)=\alpha$.
\end{proposition}

 Note that the complex Wrona metric is a $U(n)$-invariant complex Finsler metric which is singular on the set $\varOmega=\{(z,v)\in T^{1,0}D|z=\lambda v,\lambda \in\mathbb{C}\}$. It can be rewritten as
 $F(z,v)=r\phi(t,s)$ with $\phi(t,s)=\frac{1}{t-s}$.

 The following theorem shows that for any $U(n)$-invariant complex Finsler metrics defined on $\mathbb{C}^n$ their real geodesics share the same property as the complex Wrona metric. More precisely, we have
\begin{theorem}\label{th7}
  Suppose that $F(z,v)=\sqrt{r\phi(t,s)}$ is a $U(n)$-invariant complex Finsler metric defined on $\mathbb{C}^n$ which is normalized such that $\phi(1,0)=1$. Let $0<\alpha<\frac{\pi}{2}$ and $\gamma(\tau)(0\leq \tau\leq\alpha)$ be a real geodesic of $F$ on the unit sphere $\pmb{S}^{2n-1}\subset\mathbb{C}^n$ which is parameterized by arc length. Then
$$
L(\gamma)=\alpha.
$$
\end{theorem}
\begin{proof}
The proof essentially goes along the same lines as that in \cite{SR}.
In fact, let $z,w\in\pmb{S}^{2n-1}$ such that $\langle z,w\rangle=0$ and write $\gamma(\tau)=z\cos\tau+w\sin\tau, 0\leq \tau\leq \alpha$. Set $\tau_i=\frac{\alpha}{m}i,0\leq i\leq m$. Then
\begin{eqnarray*}
\langle\gamma(\tau_{i+1}),\;\gamma(t_i)\rangle&=&\cos\frac{\alpha}{m},\quad \left\|\gamma(\tau_{i+1})-\gamma(\tau_i)\right\|^2=2\Big(1-\cos\frac{\alpha}{m}\Big)
\end{eqnarray*}
and
\begin{eqnarray*}
F(\gamma(\tau_i),\;\gamma(\tau_{i+1})-\gamma(\tau_i))
&=&\sqrt{\|\gamma(\tau_{i+1})-\gamma(\tau_i)\|^2\phi\left(1,\;\frac{|\langle \gamma(\tau_i), \gamma(\tau_{i+1})-\gamma(\tau_i)\rangle|^2}{\|\gamma(\tau_{i+1})-\gamma(\tau_i)\|^2}\right)}\\
&=&2\sin\frac{\alpha}{2m}\sqrt{\phi\Big(1,\;\sin^2\frac{\alpha}{2m}\Big)}.
\end{eqnarray*}
Consequently
\begin{eqnarray*}
L(\gamma)=\lim_{m\rightarrow \infty}\left[2m\sin\frac{\alpha}{2m}\sqrt{\phi\Big(1,\;\sin^2\frac{\alpha}{2m}\Big)}\right]=\alpha\phi(1,0)=\alpha,
\end{eqnarray*}
this completes the proof.
\end{proof}
\begin{remark}
It is obvious that the function $\phi(t,s)$ corresponding to the complex Wrona metric satisfies $\phi(1,0)=1$. On the other hand, any $U(n)$-invariant complex Finsler metric $F=\sqrt{r\phi(t,s)}$ defined on $\mathbb{C}^n$ can be normalized such that $\phi(1,0)=1$, since $F$  can multiplied by a positive constant $\frac{1}{\phi(1,0)}$ if necessary. We also point it out that the Bergman metric $F=\sqrt{r\phi(t,s)}$ with $\phi(t,s)=\frac{1}{1-t}+\frac{s}{(1-t)^2}$ defined on the open unit ball $\mathbb{B}^n\subset\mathbb{C}^n$ does not satisfy the condition in the above theorem since $\phi(1,0)=\infty$. Theorem \ref{th7} actually implies that the great circles are real geodesics of $F$ when restricted to the unit sphere $\pmb{S}^{2n-1}$ in $\mathbb{C}^n$.
\end{remark}
\vskip0.4cm
{\bf Acknowledgement:}\ {\small This work is supported by the National Natural Science Foundation of China (Grant No. 11671330), the Nanhu Scholars Program for Young Scholars of XYNU, and the Scientific Research Fund Program for Young Scholars of XYNU (No. 2017-QN-029).

\end{document}